\documentclass[a4paper,12pt]{amsart}
\usepackage[utf8]{inputenc}
\usepackage{lmodern}
\usepackage{amsthm}
\usepackage{amssymb}
\usepackage{amsmath,amscd}
\usepackage{eufrak}

\title[Manin-Mumford and the maximal compact]{A Manin-Mumford theorem for the maximal compact subgroup of a universal vectorial extension of a product of elliptic curves}

\newtheorem{lem}{Lemma}[section]
\newtheorem{thm}{Theorem}[section]

\newtheorem{prop}{Proposition}[section]
\newcommand{\C}{\mathbb{C}}
\newcommand{\Q}{\mathbb{Q}}

\newcommand{\Z}{\mathbb{Z}}
\newcommand{\G}{\mathbb{G}}

\newcommand{\R}{\mathbb{R}}

\newcommand{\E}{\mathbb{E}}
\newcommand{\Ext} {\text{Ext}}

\newcommand{\X} {\mathcal{X}}
\renewcommand {\d}[2] {\frac{\partial #1}{\partial #2}}

\author{Gareth Jones}
\address{School of Mathematics, University of Manchester, Oxford Road, Manchester, M13 9PL, UK.} \address{gareth.jones-3@manchester.ac.uk}
\author{Harry Schmidt}
\address{harry.Schmidt@manchester.ac.uk}
\subjclass[2010]{Primary: 14H52, 14L10 Secondary: 14P10, 33E05}
\begin{document}

\maketitle
\begin{abstract} We study the intersection of an algebraic variety with the maximal compact subgroup of a universal vectorial extension of a product of elliptic curves. For this intersection we show a Manin-Mumford type statement. This answers some questions posed by Corvaja-Masser-Zannier which arose in connection with their investigation of the intersection of an algebraic curve with  the maximal compact subgroup of various algebraic groups.  In particular they proved that these intersections are finite for universal vectorial extensions of elliptic curves.  Using Khovanskii's zero-estimates combined with a stratification result of Gabrielov-Vorobjov and recent work of the authors, we obtain effective bounds for this intersection that only depend on the degree of the algebraic variety and the dimension of the group. As a corollary, we obtain new uniform results of Manin-Mumford type for additive extensions of certain abelian varieties.
\end{abstract}
\section{Introduction}			

The Manin-Mumford conjecture predicts that the intersection of a subvariety of a commutative algebraic group with the torsion points of the group is, in a precise sense, controlled by the group structure. This conjecture, was proved by Raynaud \cite{Raynaud1}, \cite{Raynaud2}, and the general form was proved by Hindry \cite{Hindry}. There are now several proofs, see for instance \cite{Hrushovski}, \cite{PZ}. More recently, Corvaja, Masser and Zannier \cite{CMZ} investigated the intersection of a subvariety with the euclidean closure of the torsion points in a commutative algebraic group over the complex numbers. This euclidean closure forms the maximal compact subgroup of the algebraic group. They considered in particular the case of additive extensions of elliptic curves.
												
For an abelian variety $A$ over $\C$ of dimension $g$, there is an extension $\E(A)$ of $A$ by the vector group $\G_a^g$ such that every other extension of $A$ by a vector group is a pushout of $\E(A)$. This extension is unique up to isomorphism and called the universal vectorial (or additive) extension  of $A$. For each abelian subvariety $B$ of $A$ there exists a universal vectorial extension $\E(B)$ of $B$ that is contained in $\E(A)$. This $\E(B)$ is then unique. These groups $\E(B)$ together with their translates are closed under intersections. We will focus on the case that our abelian variety is  a product of elliptic curves.

In what follows we will identify all our varieties with their complex-valued points. 

Our main result is an extension of a result due to Corvaja, Masser and Zannier \cite[Theorem 1]{CMZ}.
											
\begin{thm}\label{theorem} Let $A$ be a product of $g$ elliptic curves and $\E(A)$ a universal vectorial extension of $A$. Let $C$ be the maximal compact subgroup of $\E(A)$.   Further let $V \subset \E(A)$ be an algebraic subvariety of dimension at most $g$. Then there exists a natural number $N$ such that
 \begin{align*}
V \cap C \subset \bigcup_{i = 1}^N(t_i + \E(B_i))
\end{align*}
where the $B_i$ are proper abelian subvarieties of $A$ and $t_i \in C$.
\end{thm}

Consider the case $g=1$. Then we have $\E(A)$ the universal vectorial extension of an elliptic curve $A$ and $V\subseteq \E(A)$ is a curve. In this case the theorem says that there are at most finitely many points on $V$ lying in the maximal compact subgroup of $\E(A)$. This is the result of Corvaja, Masser and Zannier mentioned above. 
Their paper was the main motivation for our work.\\

We could also ask similar questions about other groups, for instance $\G_m^2$. In fact, Corvaja, Masser and Zannier \cite[p.228]{CMZ} already pointed out that there are examples of curves in $\G_m^2$ that have infinite intersection with $S^1\times S^1$ but that are not contained in a translate of an algebraic subgroup. 
 In this connection it is worth noting that the maximal compact subgroup of $\G_m^2$ is semialgebraic, whereas in the case of our $\E(A)$ above, the maximal compact subgroup is not semialgebraic. As further evidence of the transcendental nature of the maximal compact subgroup, we note that in the case that an elliptic curve is defined over the algebraic numbers, the only algebraic points on the maximal compact subgroup of its universal extension are the torsion points.
This was shown by Bertrand for real points \cite[Théorème 3.1]{Bertrandmaximal} and by Bost and Künnemann \cite[Theorem 3.1.2]{BK} in general.\\

The requirement in our theorem that $\dim V$ be at most $g$ appears at first sight undesirable. But this condition is necessary. This can be seen as follows. Let $\pi$ be the projection of $\E(A)$ to $A$ and let $V^{\pi}$ be  an irreducible curve in $A$. The inverse image $V = \pi^{-1}(V^{\pi})$ is an algebraic variety of dimension $g+1$ and for each point $p$ of $V^{\pi}$ we can find a point in $C$ that projects down to $p$. Thus the intersection of $V$ with $C$ is infinite and if $V^{\pi}$ does not lie in a translate of an abelian subvariety the intersection $V\cap C$ is also not contained in a finite union of translates of universal vectorial extensions strictly contained in $\E(A)$.\\

We also point out that Theorem \ref{theorem} implies a theorem about all additive extensions of $A$. In order to formulate it, we recall \cite{Brion} that an algebraic variety is anti-affine if the ring of regular functions is trivial. Each additive extension $G$ of $A$ is isogenous to a product $G_{\text{anti}} \times \G_a^k$ where $G_{\text{anti}}$ is anti-affine and $k$ is a nonnegative integer. The maximal compact subgroup of $G_{\text{anti}}\times \G_a^k$ is $C_{\text{anti}}\times \{0\}$ for $C_{\text{anti}}$ the maximal compact subgroup of $G_{\text{anti}}$. So in the investigation of intersections of algebraic varieties with the maximal compact subgroup we can restrict our attention to anti-affine groups.
   
\begin{thm}\label{additiveextensions} Let $G$ be an anti-affine extension of a product of elliptic curves  $A$ by $\G_a^l$. Let $V$ be an algebraic subvariety of $G$ of dimension at most $l$ and $C_G$ the maximal compact subgroup of $G$. There exists a natural number $N$ such that 
\begin{align*}
V\cap C_G \subset \bigcup_{i=1}^N(t_i + H_i), 
\end{align*} 
where the $H_i$ are additive extensions of proper abelian subvarieties of $A$ contained in $G$.
\end{thm}
The condition on $G$ to be anti-affine ensures that $l\leq g$. Arguing as above we see that the condition $\dim(V) \leq l$ is necessary.  \\

To state our next results it is convenient to fix a model for $\E(A)$. For each elliptic curve $E$ over $\C$, we fix a Weierstrass model with invariants $g_2$ and $g_3$. There is a corresponding projective model of $\E(E)$, described on page 245 of \cite{CMZ}. It is given by $\E(E)=\overline{\E(E)}\setminus L$, with $\overline {\E(E)}\subseteq \mathbb{P}^4$ the projective surface defined by
\begin{align}
X_0X_2^2 =  4X_1^3 -g_2X_0^2X_1 - g_3X_0^3, ~~  X_0X_4 - X_2X_3 = 2X_1^2, \label{model}
\end{align}
and $L$ the line defined by $X_0=X_1=X_2=0$. Using this embedding we can embed $\E(A)$ in multiprojective space $\left( \mathbb{P}_4\right)^g$. This embedding comes with a notion of degree. We use a rather simple-minded definition. For an algebraic variety $V$ in $\E(A)$ we define its degree of definition 
to be the minimal number $\delta$ such that its Zariski-closure $\overline{V}^{Zar} \subset (\mathbb{P}^4)^g$ is defined by  multiprojective polynomials of multi-degree at most $(\delta, \dots, \delta)$. Corvaja, Masser and Zannier showed \cite[Theorem 6]{CMZ} that for $A =E$ an elliptic curve, the bound on the cardinality of the intersection of $V$ with the maximal compact subgroup can be taken to depend only on the elliptic curve and the degree of $V$. They also asked if this result could be made effective. Here, using our recent work on pfaffian definitions of elliptic functions \cite{ourpaper}, we answer this positively and also show that the bound can be taken independently of $A$. In fact, we obtain the following effective uniform version of Theorem \ref{theorem}.
\begin{thm} \label{quant} Let $V$ be as in Theorem \ref{theorem} with degree $\delta$. Then there exist effectively computable constants $c,c'$ depending only on $\dim(A)$ such that  Theorem \ref{theorem} holds with 
\begin{align*}
N\leq c\delta^{c'}.
\end{align*}

\end{thm}

If a translate $t_i +H_i$ in Theorem \ref{theorem} contains a torsion point then we can choose $t_i$ to be a torsion point. So Theorem \ref{quant} implies a uniform and effective result of Manin-Mumford type for $\E(A)$ under the assumption that the variety considered has dimension at most the dimension of the product of elliptic curves $A$. In the case that $V$ is a surface we obtain a uniform and effective version of Manin-Mumford (see Theorem \ref{surface}).

 This is somewhat similar to the results obtained by Hrushovksi and Pillay \cite{HP}, which were strengthened by Binyamini \cite{Bi}. They obtain explicit uniform bounds when intersecting subvarieties of semiabelian varieties all defined over $\overline{\Q}$ with finite rank subgroups. Their bound only applies to transcendental points. Our bound applies to all points, including $\overline{\Q}$-rational points, but we consider additive extensions of elliptic curves, rather than semiabelian varieties.  By the same argument our result also implies a uniform and effective version of Mordell-Lang type for finite rank subgroups of $C$.\\
 
If we restrict our attention to isolated points in the intersection $V\cap C$ we can avoid the use of the stratification theorem of Gabrielov-Vorobjov \cite{GVstrat} and obtain an explicit bound.

\begin{thm}\label{iso} Let $V, C$ be as  in Theorem \ref{theorem}. The number $N_{\text{iso}}$ of isolated points in $V\cap C$ is bounded by
\begin{align*}
N_{\text{iso}}\leq 2^{42g^2+126g }g^{30g}\max\{ 3,\delta\}^{21g}.
\end{align*}
\end{thm}
We will later show (see Theorem \ref{surface}) that, for  $V$ an irreducible surface, $N_{\text{iso}}$ is also a bound for $N$ in Theorem \ref{theorem}. And for $V$ a curve all points in $V\cap C$ are isolated by \cite[Theorem 1]{CMZ}.
This has some concrete consequences for families of polynomial Pell equations. If we consider  $D=X^3Q$ where $Q\in \C(T)[X] $ is a polynomial of degree 3 with coefficients in the function field of a complex curve $T$, then we can think of the equation

\begin{align}\label{pell}
A^2-D_tB^2=1; ~~A,B\in \C[X], B\neq 0
\end{align}
where $D_t$ is the specialization of $D$ at (a suitable) $t \in T(\C)$, as a family of Pellian equations parametrized by $T$. Families of this kind were studied extensively by Masser and Zannier (see for example \cite{simple}) and they introduced a method to study the qualitative behavior of such families of equations. Now, if the family of elliptic curves defined by $Y^2=Q$ is isotrivial and \eqref{pell} does not have a generic solution then  Theorem \ref{iso} provides an explicit bound, depending only on the degree of $T$, for the number of $t \in T(\C)$ such that \eqref{pell} has a solution (see  \cite[Proposition 1.2]{Pell} and its proof in Section 3 of \cite{Pell}).\\

We expect our method to extend to certain real subgroups of universal vectorial extensions of abelian varieties. This will be carried out in later work. However the effectivity is then not so clear. We also expect the method to work in other situation such as intersections of algebraic varieties  with certain real subtori of abelian varieties. \\

%Even more generally one would expect that $g$  points on $\mathcal{C}$ indpendent over $E$ are algebraically independent but this seems  out of reach. \\

This is how the rest of this article is organized. In the next section we recall some basic facts about endomorphisms and isogenies between elliptic curves and their universal additive extensions. We then prepare the setting for the proof of Theorem \ref{theorem} and reduce it to a proposition. Then in Section \ref{transcendence} we prove some lemmas in preparation for the proof of Theorem \ref{theorem}. In Section \ref{proof} we prove Theorem \ref{theorem} and Theorem \ref{additiveextensions}. Finally in Section \ref{pfaffian} we show how to derive the explicit bound for $N_{\text{iso}}$ in Theorem \ref{theorem} from \cite{ourpaper} and Khovanskii's zero-estimates \cite[Corollary 3.3]{GV} and the effective bound for $N$ using also the stratification theorem of Gabrielov-Vorobjov.   \\

%Of course the article \cite{CMZ} of Corvaja, Masser and Zannier  was the main motivation for this work.
\section{Additive extensions of abelian vareities}\label{isogenies}
In this section we recall some facts about universal vectorial extensions and elliptic curves. Our main source is Serre's book \cite{Serre}. \\

Given a commutative algebraic group $G$, we denote by $T_O(G)$ the tangent space at the identity element $O$ of $G$. Suppose $A$ is an abelian variety. By \cite[Theorem 7, VII]{Serre} there is a canonical isomorphism  $\Ext^1(A,\G_a) \rightarrow H^1(A, \mathcal{O}_A)$ between the the group of extensions of $A$ by $\G_a$ and the group of sheaves on $A$ up to isomorphism. It is known that $H^1(A, \mathcal{O}_A)$ is isomorphic to $T_O(A^{\vee})$  where $A^{\vee}$ is the dual of $A$. For this latter fact see \cite[p.37, (e)]{Milne} and \cite[Proposition 2.1, III]{Milne}.  Thus  for any fixed $l \geq 1$ we obtain an isomorphism between $\Ext^1(A, \G_a^{l}) = \Ext^1(A,\G_a)^{l}$ and  $T_O((A^{\vee})^{l})$.  Any homomorphism $\varphi:A \rightarrow B$  induces a dual morphism $\varphi^{\vee}:B^{\vee} \rightarrow A^{\vee}$
and we define the pullback $\varphi^*:\Ext^1(B,\G_a^{l}) \rightarrow \Ext^1(A,\G_a^{l})$ such that  $\varphi^*$ corresponds to $(d\varphi^{\vee})^l:T_{O}((B^{\vee})^{l}) \rightarrow T_{O}((A^{\vee})^{l})$. \\
Let $G_1 \in \Ext^1(A, \G_a^{l_1}), G_2 \in \Ext^1(B, \G_a^{l_2})$ and $\varphi: G_1 \rightarrow G_2$ be a morphism. Then $\varphi$ gives rise to  exactly one pair $ (\varphi_{lin}, \varphi_{ab}) $ with $\varphi_{lin}:\G_a^{l_1} \rightarrow \G_a^{l_2}$, $\varphi_{ab}: A \rightarrow B$ such that $\varphi_{ab}^*(G_2) = (\varphi_{lin})_*(G_1)$. See \cite[p.162-163]{Serre} for details of these construction. \\

Now we can define a universal vectorial extension $\E(A)$ of an abelian variety $A$. By the above, a basis $b_1,\ldots,b_g$ of $T_O(A^\vee)$ determines an extension class in $\Ext^1(A,\mathbb{G}^g_a)$ which we denote by $\E(A)$. From the construction it is easy to see that $\E(A)$ is unique up to isomorphism since another choice of the basis  leads to an automorphism of $\G_a^g$. Moreover it satisfies a universal property.  For every extension $G$ of $A$ by $\G_a^l$ there exist unique maps $\gamma_1,\gamma_2$ that make the following diagram commute

\begin{align}\label{diagram}
\begin{CD}
0 @>>> \G_a^g @> >> \E(A) @>\pi >> A@ >>> 0 \\
 @.         @VV\gamma_1V @VV\gamma_2 V @VV id V @. \\
0@>>> \G_a^l @> >> G @> >> A@ >>> 0
\end{CD}
\end{align}

Using this universal property, if $A = A_1\times A_2$ then $\E(A) = \E(A_1)\times \E(A_2)$. Now given a morphism $\psi: A \rightarrow B$ we can find a unique lift $\tilde{\psi}: \E(A) \rightarrow \E(B)$ such that $\tilde{\psi}_{ab} = \psi$ as follows. We consider $\psi^*(\E(B)) \in \Ext(A,\G_a^g)$ which comes with a unique $\chi: \psi^*(\E(B)) \rightarrow \E(B)$ and the diagram above provides us with a unique morphism $\gamma_2: \E(A) \rightarrow \psi^*(\E(B))$ and we can set $\tilde{\psi} = \chi\circ \gamma_2$.  This provides us with a unique lift of $\psi$. \\

Before we proceed with the preparations for the proof of Theorem \ref{theorem} we make some more comments on isogenies and endomorphisms.\\
For elliptic curves $E_1$ and $E_2$ with fixed models we fix bases for their periods $(\omega_{1,1},\omega_{1,2})$ and $(\omega_{2,1},\omega_{2,2})$ in $\C^2$, respectively. Here and below we always assume that bases for periods are chosen such that the quotient of the second period in the basis by the first period in the basis lies in the upper halfplane.
 The differential $(d\psi)_O$ of an isogeny $\psi: E_1 \rightarrow E_2$ at $O$ acts by multiplication by a complex number $\alpha$ on $T_OE_1 = T_OE_2 = \C$. This action is such that  
\begin{align*}
\alpha (\omega_{1,1},\omega_{1,2})=(\omega_{2,1},\omega_{2,2}) \rho (\alpha)
\end{align*}
where $\rho(\alpha) \in GL^+_2(\Q)\cap M_2(\Z)$ (here $M_2(\Z)$ is the algebra of matrices with integer coefficients). For the universal extensions of the elliptic curves $\E(E_1)$  we get a period matrix
\[ 
\begin{pmatrix}
\omega_{1,1} & \omega_{1,2} \\
\eta_{1,1} &\eta_{1,2}
\end{pmatrix} \in  GL_2(\C)
\]
with 
$\eta_{1,i}$ the quasiperiod associated to $\omega_{1,i}$. This period matrix is such that the kernel of the exponential map of $\E(E_1)$ is given by 
\[
\begin{pmatrix}
\omega_{1,1} & \omega_{1,2} \\
\eta_{1,1} &\eta_{1,2}
\end{pmatrix} \Z^2.
\]
There is a similar period matrix for $E_2$, for which we use the natural notation. For all this see \cite[3.6]{CMZ}. For each elliptic curve $E$ we fix once and for all a period matrix $P_E$ as described above. We will sometimes omit the dependence on $E$. 
  The differential of the lift $\tilde{\psi}$ of $\psi$ (as defined in the paragraph following \eqref{diagram}) acts on $ \C^2$ as multiplication by a matrix $\lambda (\alpha) \in GL_2(\C)$ from the left. For the period matrices, this action is such that
\begin{align}\label{switch1}
\lambda (\alpha) \begin{pmatrix}
\omega_{1,1} & \omega_{1,2} \\
\eta_{1,1} &\eta_{1,2}
\end{pmatrix} =\begin{pmatrix}
\omega_{2,1} & \omega_{2,2} \\
\eta_{2,1} &\eta_{2,2}
\end{pmatrix}\rho(\alpha).
\end{align}     

In what follows we will write $[\alpha]$ for the isogeny  $[\alpha]:\E(E_1)\rightarrow \E(E_2)$ such that $d([\alpha]_{ab})_O = \alpha$.  It is well-known that the endomorphism ring of an elliptic curve is either $\Z$ or an order in a quadratic imaginary field $K$. After tensoring with $\Q$ the relation (\ref{switch1}) defines a map $\rho:K \rightarrow M_2(\Q)$ with image lying in $GL_2^+(\Q)\cup \{0\}$.\\%  In fact this  just defines an isomorphism between $Res_{K/\Q}(\Q)$ and  $K^*$ where $Res_{K/\Q}$ is the Weil-restriction of $K$ to $\Q$. \\

%Since we will make repeated use of complex conjugation we give it a name. We denote the complex conjugation map by $h:\C \rightarrow \C$ and extend this to complex analytic sets and manifolds. For the image of an object $U$ by $h$ we write $U^h$. For example it is clear that to a complex manifold $M$ we can define its complex conjugate $M^h$ by applying $h$ to its atlas and conjugating its transition functions by $h$.    \\

We will frequently use complex conjugation. The sets and varieties involved in our proof will all be subsets of powers of $\C^2\times \mathbb{P}^4$. Given a subset $V$ of  $\C^2\times \mathbb{P}^4$ we write $V^h$ for the image of $V$ under the usual complex conjugation map given on the two factors by
\[
(z_1,z_2)\mapsto (\overline{z}_1,\overline{z}_2)
\]
and
\[
(z_0:z_1:z_2:z_3:z_4)\mapsto (\overline{z}_0:\overline{z}_1:\overline{z}_2:\overline{z}_3:\overline{z}_4).
\]
And we will also use this notation for subsets of powers of $\C^2\times \mathbb{P}^4$.

We pause to note that if $E_1$ is isogenous to the complex conjugate $E_2^h$ of $E_2$ then 
\begin{align}\label{switch2}
\lambda(\alpha) \begin{pmatrix}
\omega_{1,1} & \omega_{1,2} \\
\eta_{1,1} &\eta_{1,2}
\end{pmatrix} = \begin{pmatrix}
\overline{\omega}_{2,1} & \overline{\omega}_{2,2} \\
\overline{\eta}_{2,1} &\overline{\eta}_{2,2}
\end{pmatrix}S
\end{align}
where $S \in GL_2^-(\Q)\cap M_2(\Z)$. That the determinant of $S$ is negative follows from the relation \[
\alpha(\omega_{1,1},\omega_{1,2}) =(\overline{\omega}_{2,1},\overline{\omega}_{2,2}) S.
\]
  The relations $(\ref{switch1})$ and $(\ref{switch2})$ are crucial for the proof of Theorem \ref{theorem} as they allow us to switch the action of the isogenies from the left to the right.
Note that we can write $\E(A^h) = \E(A)^h$ which we will repeatedly do.

\section{Setting for the proof of Theorem \ref{theorem}}\label{setting}
%In this section we prepare the setting for the proof of Theorem \ref{theorem}.
As in Theorem \ref{theorem} we set $A$ to be a product of elliptic curves, $C$ the maximal compact subgroup (of the complex-valued points) of $\E(A)$. Given an elliptic curve $E$ we write $C_E$ for the maximal compact subgroup of $\E(E)$. For the proof of Theorem \ref{theorem}  we use several properties of the exponential map
 \begin{align*}
 \exp_{\E(E)}:\C^2 \rightarrow \E(E),\\
 \end{align*}
where we have fixed an identification $T_O\mathbb{E}(A) = \C^2$. 
 
We note that
\begin{align*}
\exp_{\E(E)}(P\R^2) = C_E
\end{align*}
and we see that $C_E$ is a real analytic manifold of dimension 2. %It can be shown that $C$ is not a real-algebraic set which stands in contrast to the situation in $\G_m^2$
\\

As isogenies of complex algebraic groups send their respective maximal compact subgroups to each other, the statement of the theorem is isogeny invariant, so we can assume that 
\begin{align}\label{product}
A = E_1^{g_1}\times \cdots \times E_n^{g_n}
\end{align}
 where no pair of $E_i, E_j$ with $ i \neq j$ are isogenous and we then have that $\E(A) = \E(E_1)^{g_1}\times \dots\times \E(E_n)^{g_n}$. 

The set $C$ is a compact subanalytic subset of $\E(A)$ (viewing the latter as a real-analytic manifold), and $V$ is algebraic. So $V \cap C$ has only finitely many connected components (for instance by analytic cell decomposition in the structure $\mathbb{R}_{an}$ for which one can see for example  \cite[0.4]{Miller}). We will show that each such component is contained in a finite union of translates of universal vectorial extensions that is strictly contained in $\E(A)$.\\
 For points this is trivial. So it is enough to consider a  component of $V \cap C$ of dimension $d\geq 1$. We pick a smooth point of this component and $U_0\subset \C^{2g}$ such that $\exp_{\E(A)}(U_0) = U$ is an open neighbourhood  of our smooth point in $V \cap C$. (The existence of a smooth point is well-known, and follows for instance from analytic cell decomposition in the structure $\mathbb{R}_{an}$  \cite[0.4]{Miller}.) We can assume that $U_0$ is parametrized by real analytic functions
\begin{align*}
 (z_1,w_1, \dots,z_g, w_{g}): B_0 \rightarrow \C^{2g}
\end{align*}
where $B_0$ is an open ball in $\R^d$. We define functions 
\[
p_1,q_1, \dots, p_{g}, q_{g}: B_0 \to \R
\]
 by
\begin{align}\label{betti}
 \left(\begin{matrix} z_{k} \\w_{k} \end{matrix} \right) =  \begin{pmatrix}
\omega_{k,1} & \omega_{k,2} \\
\eta_{k,1} &\eta_{k,2}
\end{pmatrix} \left( \begin{matrix} p_k \\ q_ k\end{matrix} \right) % p_k(\omega_{1k},-\eta_{1k}) +  q_k(\omega_{2k}, -\eta_{2k})
,~~ k=1, \dots ,g
\end{align}
where
$$
 \begin{pmatrix}
\omega_{k,1} & \omega_{k,2} \\
\eta_{k,1} &\eta_{k,2}
\end{pmatrix}
$$
  is the period matrix of the $k$-th elliptic curve in the product of $A$. Note that the functions defined by \eqref{betti} do indeed take real values on $B_0$ as $U$ is contained in $C$. 
%These functions are real-analytic on $B_0$. And they take real values, as $U$ is contained in $\mathcal{C}$. 
We call these functions the Betti coordinates of $U_0$. For a study of similar Betti maps in a different context,  and some discussion of the terminology and of related concepts, see \cite{CMZ2}.\\

Complex conjugation is a continuous isomorphism between $\E(A)$ and $\E(A^h)$ and so $C^h$ is the maximal compact subgroup of the group $\E(A^h)$.  For the exponential maps we have

\begin{align*}
\overline{\exp_{\E(A)}(z_1,w_1, \dots, z_g,w_g)} = \exp_{\E(A^h)}(\overline{z}_1, \overline{w}_1, \dots, \overline{z}_g,\overline{w}_g).
\end{align*}

Of course $U_0^h$ is parametrized by $\tilde{z_1},\tilde{w_1}, \dots,\tilde{z_g} ,\tilde{w_{g}} $ with
\begin{align}\label{conjbetti}
\left(\begin{matrix} \tilde{z}_{k} \\ \tilde{w}_{k} \end{matrix} \right) = \begin{pmatrix}
\overline{\omega}_{k,1} & \overline{\omega}_{k,2} \\
\overline{\eta}_{k,1} &\overline{\eta}_{k,2}
\end{pmatrix}\left(\begin{matrix}p_k \\ q_k \end{matrix}\right), k=1,\dots, g,
\end{align}
as on $B_0$, the functions $p_1,q_1, \dots,p_g,q_g$ are real-valued. This will be an essential ingredient of our argument. \\

We pause to note that in what follows we will always work with the fibred product $\Delta(U_0\times U_0^h)\times\Delta( U\times U^h) \subset U_0\times U_0^h \times\exp_{\E(A)\times \E(A)^h}(U_0\times U^h_0)$.  Here we write $\Delta(S_1\times S_2) = \{(x,y) \in S_1\times S_2: x = \overline{y}\}$ for complex analytic sets $S_1, S_2$.  This will  be clear throughout as we work with a fixed parametrization (namely by the real ball $B_0$ above)  of our sets.

Now suppose that $U$ is contained in a universal vectorial extension strictly contained in $\E(A)$. If $M$ is a connected real-analytic manifold containing $U$ and of the same dimension as $U$, then $M$ too is contained in the same universal vectorial extension, by analytic continuation. By analytic cell-decomposition in the structure $\mathbb{R}_{\text{an}}$, the intersection $V\cap C$ is a finite union of analytic cells (see for instance \cite[0.4]{Miller}). So, in order to establish Theorem \ref{theorem}, it is sufficient to show the following.

\begin{prop}\label{prop} The set $U$ is contained in a universal vectorial extension strictly contained in $\E(A)$.
\end{prop}

We now complexify the functions $p_1,q_1,\dots,p_g,q_g $ and so by \eqref{betti} and \eqref{conjbetti} also obtain complex extensions of $z_1,w_1, \dots,z_g,w_{g}, \tilde{z_1},\tilde{w_1},\dots,\tilde{z_g},  \tilde{w_{g}}$. Since the image of the real-analytic map \[
\exp_{\E(A)\times \E(A^h) }(z_1,w_1, \dots,z_g,w_{g}, \tilde{z_1},\tilde{w_1},\dots,\tilde{z_g},  \tilde{w_{g}})
\]
lies in the complex algebraic variety $V\times V^h$ the same is true of the image of the complexification. So writing 
\[
x = (z_1,w_1, \dots,z_g,w_{g}, \tilde{z_1},\tilde{w_1},\dots,\tilde{z_g},  \tilde{w_{g}})
\]
and
\[
y = \exp_{\E(A)\times\E(A^h)}(x)
\]
we have
\[
\text{trdeg}_\C(y) \le 2\dim (V).
\]
And by \eqref{betti} and \eqref{conjbetti} we have 
\begin{align}\label{tdx}
\text{trdeg}_\C(x) \leq \text{trdeg}_\C(p_1,\ldots,p_g,q_1,\ldots,q_g) \leq 2g.
 \end{align}
Combining these we see that
\[
\text{trdeg}_\C(x,y)\le 2g+2\dim (V).
\]
Now we apply \cite[Proposition 1.b]{Bertrand} which states that 

\begin{align}
\text{trdeg}_\C(x,y)\ge \dim{H_y}+1.\label{ineq}
\end{align}
for the smallest translate of a universal vectorial extension $H_y$ containing $y$.

Combining the last two inequalities, and using our crucial assumption that $\dim (V)\le g$, we see that
\[
\dim{H_y}\le 4g-1.
\]
But the dimension on the left is even, and so we have
\[
\dim{H_y}\le 4g -2.
\]

It follows that $\Delta(U\times U^h)$ is contained in a translate of a  universal vectorial extension strictly contained in $\E(A)\times \E(A^h)$.  Translating $V$ by a point in $C$ we can assume that  $U$ contains the origin. So in what follows $\Delta(U \times U^h)$ is contained in a universal vectorial extension $H$ in $\E(A)\times\E(A^h)$ of positive codimension and we assume that $H$ is minimal with this property.\\

This is enough to prove Theorem \ref{theorem} in certain cases. For suppose that there is no elliptic curve in the product defining $A$ that is isogenous to the complex conjugate of another curve in the product. By a theorem of Kolchin (see \cite[Lemma 7]{MW}) the projection of $H$ to $A\times A^h$ factors as a product of subgroups of $A$ and $A^h$. And one of these must be a proper subgroup, from which it follows that $U$ is contained in a universal vectorial extension strictly contained in $\E(A)$, thus establishing Proposition \ref{prop} in this special case.\\

From now on we will assume that $U$ is not contained in a  universal  extension in $\E(A)$ of positive codimension. The same then holds for $U^h$ in $\E(A^h)$. \\

To ease notation  we will abbreviate universal vectorial extension by universal extension and a universal vectorial extension that is strictly contained in another by universal subextension. \\

\section{Transcendence degree}\label{transcendence}

For an elliptic curve $E$ with complex multiplication we let $K$ be the CM field of $E$. If $E$ does not have complex multiplication we set  $K$ to be equal to some imaginary quadratic field. This is only  to make the proofs uniform. The following lemma is central to the proof of Theorem \ref{theorem}. Recall from \eqref{switch1} that $\rho:K\to M_2(\Q)$ has image lying in $GL^+_2(\Q)\cup\{0\}$.

\begin{lem}\label{rank1}
Let $M$ and $\tilde{M} $ be two $r \times l$ matrices with entries $m_{st}$ and $\tilde{m}_{st}$ in $K$, respectively, where  $  s = 1, \dots, r$ and  $t = 1, \dots, l$. Suppose that $M$ has rank $r$. If $S $ is a rational  $2\times 2$ matrix with negative determinant then the $2r\times 2l$ integer matrix $\hat{M}$ defined block-wise by
\begin{align*}
\hat{M}_{st} = \left(\rho(m_{st})  +\rho(\tilde{m}_{st}) S \right)
\end{align*}
has rank at least $r$.
\end{lem}
\begin{proof}The matrix $\hat{M}$ consists of pairs of row vectors $\hat{M}^{(1)},\ldots,\hat{M}^{(r)} \in \text{Mat}_{2,2l}(\Q)$ arranged vertically, with each $\hat{M}^{(s)}$ consisting of $2\times 2$ matrices $\rho(m_{st})+\rho(\tilde{m}_{st})S$ with rational coefficients arranged horizontally as $t=1,\ldots,l$.\\

We first claim that if $\alpha_1,\ldots,\alpha_r\in K$ are such that 
\begin{align}\label{lindep}
\sum_{s=1}^r\rho(\alpha_s)\hat{M}^{(s)}=0
\end{align}
then $\alpha_1=\cdots=\alpha_r=0$. To see this, first note that if \eqref{lindep} holds, then for each $t=1,\ldots,l$ we have
\begin{align*}
0 =& \sum_{s=1}^r\rho(\alpha_s)\rho(m_{st})+\rho(\alpha_s)\rho(\tilde{m}_{st})S\\
=& X_t+Y_t S
\end{align*}
where
\[
X_t= \rho\left( \sum_{s=1}^r \alpha_s m_{st}\right)
\]
and
\[
Y_t=\rho\left( \sum_{s=1}^r \alpha_s \tilde{m}_{st}\right).
\]
So for each $t=1,\ldots,l$ we have
\[
\det (X_t)=\det(Y_t)\det(S).
\]
But for $\alpha\in K$, the determinant of $\rho(\alpha)$ is the norm of $\alpha$ from $K$ to $\Q$, and so is nonnegative, and is zero if and only if $\alpha=0$. As $\det (S)<0$, we must have $X_t=Y_t=0$.  In particular,
\[
\sum_{s=1}^r\alpha_s m_{st}=0.
\]
As $M$ has rank $r$, we must have $\alpha_1=\cdots=\alpha_r=0$, proving the claim.\\

To prove the lemma, consider the $\Q$-vector space 
\[\mathbb{U}=\{ v \in \text{Mat}_{2,2r}(\Q) : v\hat{M}=0\}.\]
 This has dimension twice the dimension of the kernel of $\hat{M}^t$, so $\dim \mathbb{U}= 2(2r-\text{rank}(\hat{M}))$. So it suffices to show that $\dim \mathbb{U}\le 2r$. Suppose that $\dim \mathbb{U}>2r$. Since $[K:\Q]=2$, the $\Q$-vector space
\[
\mathbb{V}=\left\{ \left( \rho(\alpha_1),\ldots,\rho(\alpha_r) \right) \in \text{Mat}_{2,2r}(\Q) : \alpha_1,\ldots,\alpha_r\in K\right\}
\]
has dimension $2r$. So
\begin{align*}
\dim \mathbb{U}\cap \mathbb{V} =&\dim \mathbb{U} +\dim \mathbb{V}-\dim (\mathbb{U}+ \mathbb{V})\\
>&2r+2r-4r=0.
\end{align*}
Hence there is some $v=\left( \rho(\alpha_1),\ldots, \rho(\alpha_r)\right) \in \mathbb{U} \setminus \{ 0\}$, and in particular, $v\hat{M}=0$, that is,
\[
\rho(\alpha_1)\hat{M}^{(1)}+\cdots +\rho(\alpha_r)\hat{M}^{(r)}=0.
\]
But this contradicts the claim above. 
\end{proof}
We also record the following well-known fact as a  lemma. 

\begin{lem} \label{rank2} Let $M$ be a matrix of rank $r$ with entries  $m_{st} \in K, s =1, \dots, r, t =1,\dots,l$. Then the $2r\times2l$ integer matrix $\hat{M}$ defined blockwise by
\begin{align*}
\hat{M} = (\rho(m_{st}))_{st}
\end{align*}
has rank $2r$.
\end{lem}
\begin{proof}
Since $M$ has rank $r$ there is an $r\times r$ submatrix of $M$ with non-zero determinant $D \in K$. Taking the image of this matrix by $\rho$ in $\hat{M}$ we can evaluate its determinant block-wise \cite[pp. 546-547]{Bourbaki} and obtain that it is equal to $\text{det}( \rho(D))$. This vanishes if and only if $D =0$. 

\end{proof}

We now consider the setting of Theorem \ref{theorem}, and work towards the proof of Proposition \ref{prop}. So we have $\E(A) = \E(E_1)^{g_1}\times \dots \times \E(E_n)^{g_n}$, with $E_i$ not isogenous to $E_j$ for $i\ne j$.\\ 

We will use Lemma \ref{rank1} and Lemma \ref{rank2} to bound the transcendence degree of $x$ from above (recall that $x$ is defined in the discussion following Proposition \ref{prop}). So we improve the estimate in \eqref{tdx}. We first consider the projections $\hat{H}$ of $H$ to products $\E(E_i)^{g_i} \times \E(E_i^h)^{g_i}$ such that $\E(E_i)$ and $\E(E_i^h)$ are isogenous and $\hat{H}$ is a universal subextension of $\E(E_i)^{g_i} \times \E(E_i^h)^{g_i}$. In order to ease notation we set $E_i =E$, $\E(E_i) = G$ and  $g_i = e$. We further denote by $P$ the period matrix of $E$ that we have fixed. By abuse of notation we will renumber  $p_1,q_1, \dots, p_e,q_e$ to be the Betti coordinates of the projection of the coordinates of  $U_0$  to the tangent space at the identity of $G^{e}$.

\begin{lem} \label{ii} Let $\hat{x}=(p_1, q_1,\dots,p_{e}, q_{e})$ and $\hat{H}$ be as described above. Then
\begin{align*}
\text{trdeg}_\C(\hat{x}) \leq \frac12 \dim(\hat{H}) .
\end{align*}
\end{lem}
\begin{proof} Recall from the discussion following the statement of Proposition \ref{prop} that we assume that $U$ is not contained in a universal subextension.  
 We denote by $Q_s$ the projection from $U$ to the sth copy of $G$ in the product $G^e$ and similarly by $\tilde{Q}_s$ the projection from $ U^h$ to the sth copy of  $G^h$,  for $s=1,\ldots, e$.
 Let $r = 2e - \frac12 \dim(\hat{H})$. As $U$ is not contained in a universal subextension of $\E(A)$, the group $\hat{H}$ is defined by relations of the following form
\begin{align}
\sum_{s=1}^{e}[m_{st}]\circ Q_{s} = -\sum_{s=1}^{e}[\tilde{m}_{st}] \circ [\alpha]\circ \tilde{Q}_{s},~~ t=1, \dots,r, ~~ m_{st}, \tilde{m}_{st} \in K \label{relationH}
\end{align}
where $[\alpha]$ is an isogeny from $G^h$ to $G$ and $M= (m_{st})$ and $ \tilde{M} = (\tilde{m}_{st})$ are matrices of rank $r$ with entries in $K$. This is because if the rank of either $M$ or $\tilde{M}$ is less than $r$ we find in both cases that the projection of $H$ to $G^e$ is contained in a proper algebraic subgroup of $G^e$, which in turn contradicts our assumption on $U$.   
%The  isogeny $[\alpha]$ descends to an isogeny $(\alpha)$ between $E$ and $\overline{E}$ and it acts by
%\begin{align*}
%\alpha(\overline{\omega_1},\overline{ \omega_2}) = (\omega_1, \omega_2)B
%\end{align*}
%or an integer matrix $B$ with negative determinant. 
%So $B =B(\alpha)$ as in Definition \ref{defiso} where $ \tilde{P}$ is replaced by $\overline{P}$. So
 Recall that we have
\begin{align}
\lambda(\alpha)P^h= PS\label{rl}
\end{align}
 as in (\ref{switch2}). By taking the logarithm, one can check that on $\Delta(U_0\times U_0^h)$ the relation (\ref{relationH}) translates to a relation of the following form
\begin{align*}
\sum_{s= 1}^{e}\left(\lambda(m_{st})P\left(\begin{array}{c} p_{s} \\ q_{s}\end{array}\right) + \lambda(\tilde{m}_{st})\lambda(\alpha)P^h\left(\begin{array}{c} p_{s} \\ q_{s}\end{array}\right)\right) = 0~ \mod \C^2.
\end{align*}
Now using (\ref{rl}) and (\ref{switch1}) we obtain, after multiplying by $P^{-1}$ (which exists by the Legendre relation),
\begin{align*}
\sum_{s= 1}^{e}\left(\rho(m_{st}) + \rho(\tilde{m}_{st})S\right)\left(\begin{array}{c} p_{s} \\ q_{s}\end{array}\right) = 0~ \mod \Z^2
\end{align*}
for $t=1,\dots,r$. By Lemma \ref{rank1} this implies that $\text{trdeg}_\C(p_{1},q_{1}, \dots, p_{e}, q_{e}) \leq 2e - r$.
\end{proof}

We need one further lemma to treat the projections of $H$ to products $\E(E_i)^{g_i}\times \E(E_j^h)^{g_j}, i \neq j$,  for which $E_i$ and $E^h_j$ are isogenous. Now let $\hat{H}$ be the projection of $H$ to $\E(E_i)^{g_i}\times \E(E_j^h)^{g_j}$.  We set $e = g_i, \tilde{e} = g_j, G= \E(E_i)$ and $\tilde{G}= \mathbb E(E^h_j)$. We further denote by $(p_1,q_1, \dots, p_e,q_e)$ and $(\tilde{p}_1,\tilde{q}_1, \dots, \tilde{p}_{\tilde{e}},\tilde{q}_{\tilde{e}})$ the Betti coordinates of the projection of the coordinates of $U_0$ to the tangent space at the identity of $G^e$ and $\tilde{G}^{\tilde{e}}$, respectively. 

\begin{lem} \label{ij} Let $\hat{x} =(p_1,q_1, \dots, p_e,q_e, \tilde{p}_1,\tilde{q}_1, \dots, \tilde{p}_{\tilde{e}},\tilde{q}_{\tilde{e}})$ and $\hat{H}$ be as above. Then
\begin{align*}
\text{trdeg}_\C(\hat{x}) \leq  \dim(\hat{H}) .
\end{align*}

\end{lem}
\begin{proof}
Similarly as for Lemma \ref{ii} we set $r = e +\tilde{e} - \frac12\dim(\hat{H})$. Let $(Q_1, \dots, Q_{e}$, $ \tilde{Q}_1, \dots, \tilde{Q}_{\tilde{e}})$ be as in Lemma \ref{ii}.  They satisfy relations of the following form
\begin{align}
\sum_{s=1}^{e}[m_{st}]\circ Q_{s} = -\sum_{\tilde{s}=1}^{\tilde{e}}[\tilde{m}_{\tilde{s}t}] \circ [\alpha]\circ \tilde{Q}_{\tilde{s}},~~ t=1, \dots,r, ~~ m_{st}, \tilde{m}_{\tilde{s}t} \in K \label{relation}
\end{align}
where $[\alpha]$ is an isogeny from $\tilde{G}$ to $G$. We note as in the proof of Lemma \ref{ii} that both matrices $(m_{st})$ and $ (\tilde{m}_{\tilde{s}t})$ have rank $r$ because of our assumption on $U$.
As in Lemma \ref{ii} this translates to a relation on $U_0 \times U_0^h$ of the following form
\begin{align*}
\sum_{s= 1}^{e}\rho(m_{st})\left(\begin{array}{c} p_{s} \\ q_{s}\end{array}\right) +
 \sum_{\tilde{s}=1}^{\tilde{e}} \rho(\tilde{m}_{\tilde{s}t})S\left(\begin{array}{c} \tilde{p}_{\tilde{s}} \\ \tilde{q}_{\tilde{s}}\end{array}\right)= 0~ \mod \Z^2
\end{align*}
for some integer matrix $S$.
As $(m_{st})$ has rank $r$ it follows from Lemma \ref{rank2} that $\text{trdeg}_\C(\hat{x}) \leq 2(e + \tilde{e}) -2r$.
\end{proof}

\section{Proof of Theorem \ref{theorem} and \ref{additiveextensions}}\label{proof}
%Now we  prove Theorem \ref{theorem}.
\begin{proof}[\it{Proof of Theorem \ref{theorem}}]%\textit{(Theorem \ref{theorem})}
We will prove Theorem \ref{theorem} by induction on $g + \dim(V)$. If $\dim(V) = 0$, Theorem \ref{theorem} is trivial. So Theorem \ref{theorem} holds for $g + \dim(V) = 1$. \\

For our induction argument we may assume that $\dim V \geq 1$ and that $U, U_0,x,y$ and $H$ are as in section \ref{setting}. Recall that we assume that $U$ is not contained in a universal subextension. We show that this leads to a contradiction. For each elliptic curve $E$ there is up to isogeny exactly one elliptic curve $E'$ such that $E$ is isogenous to $E'^h$. So  we may permute the set $ \{1, \dots, n\}$ such that for odd $i \leq 2m$, the curve $E_i$ is isogenous to $E^h_{i+1}$ while for $2m<i\leq 2m +m'$ the curve $E_i$ is isogenous to $E^h_i$ and the rest are not isogenous to any complex conjugate of any of $E_1, \dots, E_n$. \\

 We then write
\begin{align} \label{product}
\E(A)\times \E(A^h) = G_1 \times \dots \times G_{2m +m'} \times T
\end{align}
with $G_i = \E(E_i)^{g_i}\times\E(E_{j}^h)^{g_{j}}$ where we set $j = i+1$ for odd $ i \leq 2m$ and $j = i-1$ for even $i \leq 2m$ and finally for $2m <i \leq 2m + m'$ we set $j = i$. We denote by $\pi_G$ the projection to a subgroup $G$ appearing as a factor in the product (\ref{product}) of $\E(A)\times\E(A^h)$ and $(d\pi_{G_i})_O$ to mean the projection between the corresponding tangent spaces.\\%. {\bf Need to say what we mean by projection here.}\\

We start the proof by first investigating the transcendence degree of $x$ over $\C$.
Note that $H = \tilde{H} \times T$ for some $\tilde{H}$. This is the case since the projection of  $T$ to $A$ is a product of elliptic curves such that no elliptic curve in it is isogenous to the complex conjugate of any other in $A$. So if the image of $H$ by the projection to $T$ were strictly contained in $T$ then the projection of $H$ to $\E(A)$ would be too, contradicting our assumption that $U$ is not contained in a subextension. It follows that
\begin{align}
\dim(H) = \sum_{i=1}^{2m+m'}\dim(\pi_{G_i}(H)) + \dim(T) .
\end{align}

Applying Lemma \ref{ij}, we see that
\begin{align*}
\text{trdeg}_\C(d\pi_{G_i})_O(x) \leq \dim(\pi_{G_i}(H)).
\end{align*}
for  $i \leq 2m$. But we also note that for these $i$ the conjugate $G^h_i$ also turns up in the product $G_1\times \cdots \times G_{2m}$ and that
\begin{align*}
 \text{trdeg}_\C(d\pi_{G_i\times G_i^h})_O(x) \leq \frac12\dim(\pi_{G_i\times G^h_i}(H)).
\end{align*}
And by Lemma \ref{ii} we have
\begin{align*}
\text{trdeg}_\C(d\pi_{G_i})_O(x) \leq\frac12\dim(\pi_{G_i}(H))
\end{align*}
for $2m<i \leq 2m + m'$. The conjugate of each factor of $T$ also appears in $T$ and so
\begin{align*}
 \text{trdeg}_\C(d\pi_{T})_O(x) \leq \frac12\dim(T).
\end{align*}
Thus we deduce that
\begin{align*}
\text{trdeg}_\C(x) \leq \frac12\dim(H).
\end{align*}
If $\dim(H)<2g$ then $\dim \pi_{\E(A)}(H) <2g$ and so $U$ is contained in a universal subextension. So we assume now that $\dim(H)\geq 2g$.

Let $r' = 2g - \frac12\dim(H)$. We write $A$  as a product of (not necessarily distinct) elliptic curves $A =A_1\times \cdots \times A_g$.  We denote by $Q_1,\dots, Q_g $ the projections of $U$ to $\mathbb{E}(A_1),\dots,\mathbb{E}(A_g)$ and by $\tilde{Q}_1,\dots,\tilde{Q}_g$ the projections of $U^h$ to $\mathbb{E}(A_1^h),\dots,\mathbb{E}(A^h_g)$ respectively. There are $r'$ sums appearing on the left of  the relations (\ref{relationH}), (\ref{relation}) defining $H$. These sums give rise to a surjective group homomorphism
\begin{align*}
\sigma:\E(A) \rightarrow G'
\end{align*}
where $G'$ is a universal extension of $r'$ elliptic curves. Let $V'$ be the Zariski-closure of $\sigma(U)$ in $G'$.

We distinguish between three cases. We first assume that
\begin{align*}
\dim(V') < r' \text{ and }\dim (V')<\dim (V).
\end{align*} 														
Let $C'$ be the maximal compact subgroup of $G'$. The map $\sigma$ sends $U$ to a component of $V'\cap C'$.  As $\dim(H) \geq 2g$ it holds that $r' \leq g$ so it follows  by induction that $V'\cap C'$ is contained in a finite union of universal extensions properly contained in $G'$. It follows that $U$ is contained in a proper universal subextension of $\E(A)$ and we derive a contradiction.
Next suppose that 
\begin{align}\label{assumption}
\dim(V') \geq r'.
\end{align} 			
Let $\sigma\circ (Q_1,\dots,Q_g)$ be the composition of $\sigma$ with $(Q_1,\dots,Q_g)$. Composing this with the coordinate functions of the group generates a field  $\C(\sigma\circ(Q_1,\dots,Q_g))$ that has transcendence degree at least $r'$ over $\C$. We similarly define the fields $\C(Q_1,\dots,Q_g), \C(\tilde{Q}_1,\dots \tilde{Q}_g)$. Since  $\C(\sigma \circ (Q_1,\dots,Q_g)) \subset \C(Q_1, \dots, Q_g)\cap \C(\tilde{Q}_1, \dots, \tilde{Q}_{g})$  and both $\C(Q_1, \dots, Q_g)$ and $ \C(\tilde{Q}_1, \dots, \tilde{Q}_g)$ have transcendence degree at most $g$ over $\C$, it follows that
\begin{align*}
\text{trdeg}_\C(Q_1, \dots, Q_g,\tilde{Q}_1, \dots, \tilde{Q}_g  ) \leq 2g- r' = \frac12\dim(H).
\end{align*} 																	
Finally if $\dim(V')=\dim (V)$ then $\text{trdeg}_\C(\sigma\circ(Q_1,\dots,Q_g)) =\dim (V)$. And since both $\C(Q_1, \dots, Q_g)$ and $ \C(\tilde{Q}_1, \dots, \tilde{Q}_g)$ have transcendence degree at most $\dim (V)$, we have 
\[
\text{trdeg}_\C(Q_1, \dots, Q_g,\tilde{Q}_1, \dots, \tilde{Q}_g  ) \leq \dim (V)\le g \le \frac12\dim(H).
\]
Either way we have
\begin{align*}
\text{trdeg}_\C(x,y) \leq \dim(H).
\end{align*}																																												
As we assumed that the dimension of $U$ is at least 1  and that $H$ is minimal we deduce a contradiction from the inequality \eqref{ineq}. Thus $U$ has to be contained in a proper universal subextension of $\E(A)$.
\end{proof}
\begin{proof}[\it{Proof of Theorem \ref{additiveextensions}}]
It follows from the universal property of $\E(A)$ that there exists a group homomorphism $\gamma_2:\E(A) \rightarrow G$ as in (\ref{diagram}). Since the co-kernel of $\gamma_2$ is a vector group and $G$ is anti-affine $\gamma_2$ is surjective. The variety $V^{\gamma_2} = \gamma^{-1}(V)$ has dimension $\dim V + g - l \leq g$ and so the intersection of $V^{\gamma_2}$ with $C$ is contained in a union of translates of universal subextensions of $\E(A)$. The image of these translates by $\gamma_2$ are translates of additive extensions of proper abelian subvarieties of $A$ contained in $G$ and cover the intersection $V\cap C_G$.  
\end{proof}
\section{Effectivity and a refinement of Theorem \ref{theorem}}\label{pfaffian}

In this section we will use several properties of the exponential map. For an elliptic curve $E$ and the model of $\E(E)$ given by \eqref{model} we can take
 \begin{align*}
 \exp_{\E(E)}:\C^2 \rightarrow \E(E) \subseteq \mathbb{P}^4(\mathbb{C})
 \end{align*}
given by
\begin{align}\label{exponential}
(z,w) \rightarrow \begin{cases}
 (1:\wp(z): \wp'(z):\zeta(z) +w: \wp'(z)(\zeta(z)+w)+2\wp(z)^2)& z \notin \Lambda\\
 (0:0:1:0:w+\eta(\omega)) & z= \omega\in \Lambda
 \end{cases}
\end{align}
where $\Lambda$ is the kernel of the exponential of $E$ and $\eta(\omega)$ is the quasiperiod associated to $\omega$ (see for instance \cite[pages 252-3]{CMZ}).
To get the explicit bound in Theorem \ref{iso}, we will use our earlier work \cite{ourpaper}. We briefly recall the relevant definitions.  We say that a sequence $f_1,\ldots,f_l:U\to \R$ of analytic functions on an open set $U$ in $\R^n$ is a \emph{pfaffian chain} if, for $i=1,\ldots,l$ and $j=1,\ldots,n$ there are real polynomials $p_{i,j}$ in $n+i$ indeterminates such that
\[
\d{f_i}{x_j} (x) = p_{i,j}(x,f_1(x),\ldots,f_i(x))
\]
on $U$. We say that a function $f$ is \emph{pfaffian} if there is a polynomial $p$ such that $f$ is $p(x,f_1(x),\ldots,f_l(x))$. We say that $f$ as above has \emph{order} $l$ and degree $(\alpha,\beta)$, where $\alpha$ is a bound on the maximum of the degrees of the $p_{i,j}$ and $\beta$ is a bound on the degree of $p$. To apply the effective estimates due to Khovanskii,  and to Gabrielov and Vorobjov, the domains of the pfaffian functions involved cannot be too complicated. So we say that $U\subseteq \R^n$ is a \emph{simple domain} if $U$ is the image of a product of open intervals under an invertible affine transformation. Now let $\X\subseteq \R^n$. Suppose that $U_i\subseteq \R^n$ are simple domains, for $i=1,\ldots, L$, and that for each $i$ we have pfaffian functions $f_{i,1},\ldots,f_{i,m_i}:U_i\to \R$ with a common chain of order $r$ and degree $(\alpha,\beta)$.  And suppose that $m_i\le M$  and that
\[
\X=\bigcup_{i=1}^L \{ x \in U_i : f_{i,1}(x)=\cdots= f_{i,m_i}(x)=0\}
\]
Then we call $\X$ \emph{piecewise semi-pfaffian}, of \emph{format} $(r,\alpha,\beta, n , L,M)$. We refer to the entries of the format of a piecewise semi-pfaffian set as the $r$-entry, $\alpha$-entry, and so on. 

With the definitions out of the way, we can give the details of the proof of Theorem \ref{iso}.
\begin{proof}[\textit{Proof of Theorem \ref{iso}}]
For an elliptic curve $E$ with universal extension $\E(E)$ with exponential as in \eqref{exponential} we fix periods $\omega_1, \omega_2$  such that $\omega_2/\omega_1$ lies in the standard fundamental domain in the upper half plane and let $F^0$ be the fundamental parallelogram spanned by $\omega_1, \omega_2$ but deprived of 0 so $F^0 = \{r_1\omega_1 +r_2\omega_2: r_1,r_2 \in [0,1), r_1^2 + r_2^2 \neq 0\}$. We will work with the affine chart $\E(E)^0$ of $\E(E)$ defined by $X_0\neq0$.  The exponential of $\E(E)$ restricted to $F^0\times \C$ maps onto $\E(E)^0$, as described in (\ref{exponential}). \\
By the proof of \cite[Theorem 11]{ourpaper} the graph $\Gamma(\wp,\zeta)\subseteq \C^3$ of $(\wp,\zeta)$ where $\wp,\zeta$ are restricted to $F^0$ is a piecewise semi-pfaffian set of format $(9,9,1,6,144503,4)$. We define $b_1, b_2 $ to be the Betti coordinates of $z$ with respect to $\omega_1, \omega_2$ which are degree 1 polynomials in the real and imaginary part of $z$.  Further let $\eta_1, \eta_2$ be the quasi-periods of $\zeta$ associated to $\omega_1, \omega_2$.  We identify $\C$ and $\R^2$ in the usual way (as we shall do throughout this proof). Then using the piecewise semi-pfaffian description of $\wp$ and $\zeta$ mentioned above, we see that
\begin{equation}\label{XG}
\begin{split}
\X_E  =&\Big\{\left(z,w,X_1,X_2, X_3, X_4\right)\in \C^6:
 (z,X_1,X_3 -w) \in \Gamma(\wp,\zeta), \\
  &X_2^2 - 4X_1^3-g_2X_1-g_3 = 0,  X_4- X_3 X_2 +2X_1^2=0,  w -b_1\eta_1 -b_2\eta_2=0\Big\}
\end{split}
\end{equation}
is a piecewise semi-pfaffian set. Here the first two equations define the model of $\E(E)$ and the last makes sure that $\X_E $ is  contained in the maximal compact subgroup of $\E(E)$. As a piecewise semi-pfaffian set, $\X_E$ has format $(9,9,3,12,144503,10)$ as we have added $6$ additional equations that involve polynomials of degree at most $3$ to the definition of $\Gamma (\wp,\zeta)$. \\

We embed $\E(A) =\prod_{k=1}^g\E(E_k)$ into multiprojective space $(\mathbb{P}^4)^g$ by embedding each $\E(E_k)$ into $(\mathbb{P}^4)^g$ as in (\ref{model}). We denote by $X_{0k},X_{1k}, X_{2k}, X_{3k}, X_{4k}$ the projective coordinates of $\E(E_k)$ and   set $\E(E_k)^0$ to be the affine chart of $\E(E_k)$ defined by $X_{0k}\neq 0$.
The projection of each point in the intersection of $V\cap C$ to $\E(E_k)$ is either equal to the identity element or lies on the chart defined by $X_0\neq 0$. Thus the intersection $V\cap C$ lies in the union of the $2^g$ sets defined by these conditions.  The intersection of $V$ with such a set is given by, for each $k$, either specializing $X_{0k}=1$ or specializing $X_{0k},X_{1k}, X_{2k}, X_{3k}, X_{4k}$ to the identity element of $\E(E_k)$ in the equations defining $V$.

 For each $k=1,\ldots,g$ we form the set $\X_{E_k}$ as in \eqref{XG}. For each subset $S$ of  $\{1,\dots, g\}$ we form the cartesian product $\X_S= \prod_{k\in S}\X_{E_k}$. These products are piecewise semi-pfaffian sets of format $$(9|S|,9|S|,3,12|S|,144503^{|S|}, 10|S|)$$ with $|S|\leq g$. For each such $S$ we consider the real and imaginary part of the  polynomials defining $V$ and set $X_{0k}=1$ for $k\in S$ and $X_{0k},X_{1k}, X_{2k}, X_{3k}, X_{4k}$ equal to the identity element otherwise. The resulting polynomials have total degree bounded by $\delta$. We add those equations to the definition of $\prod_{k\in S} \X_{E_k}$ and obtain a set $V_S$ that is contained in the union of  $144503^g$  zero sets of  pfaffian functions of order $9g$ and degree $(9g,\max\{3,\delta\})$. Each isolated point of $V\cap C$ lies in the projection of exactly one connected component of such a set $V_S$ to $\E(A)$. The number $N_{\text{iso}}$ of isolated points in $V\cap C$ is thus bounded by $2^g$ times a bound for the number of connected components of $V_S$. Using  Khovanskii's bounds \cite[Corollary 3.3]{GV} we obtain
\begin{align*}
N_{\text{iso}}\leq 2^{42g^2+126g }g^{30g}\max\{ 3,\delta\}^{21g}.
\end{align*}

 \end{proof}

Now we quickly show that Theorem \ref{iso} implies an explicit bound for $N$  if $V$ has dimension at most 2.
\begin{thm} \label{surface}   For $V\subseteq \E(A)$ an irreducible surface and $g\geq 2$ the intersection $V\cap C$ is finite or $V$ is equal to a translate of a universal subextension.
\end{thm}
\begin{proof} We argue by contradiction. Suppose that $V$ is not a translate and that there is a positive dimensional component $Y$ of $V\cap C$. By Theorem \ref{theorem} we have that $Y$ is contained in a finite union of translates of universal subextensions. If the dimension of the intersection of $V$ with a translate is smaller than 2 then, by Theorem \ref{theorem} again (due to Corvaja, Masser and Zannier in this case \cite[Theorem 1]{CMZ}) the intersection is finite since every positive dimensional translate has dimension at least 2. So there has to be a translate $c+H$ for which the dimension is 2. As $V$ is irreducible, $V$ is contained in that translate. Thus $V-c$ is contained in $H$ and $Y-c$ is a positive dimensional component of $V$ intersected with the maximal compact subgroup of $H$. If $\dim(H)\ge4 $ we can iterate this argument for $V-c$ until the dimension of $H$ is equal to 2. But if $V$ is contained in a translate of dimension 2 then $V$ is equal to $H$ again obtaining a contradiction.  
\end{proof}

Finally we prove the effective version of Theorem \ref{theorem}. We will again use our earlier work but we will also need the stratification theorem of  Gabrielov-Vorobjov. We assume that the reader is familiar with the definitions from \cite{GVstrat}.
\begin{proof}[\textit{Proof of Theorem \ref{quant}}]

Here and also below the constants $c_i$ depend only on $g$, and can be computed from $g$.

Let $V_S$ be as in the proof of Theorem \ref{iso}. As we mentioned there it is a piecewise semi-pfaffian set and  the entries of its format are bounded by $c_1$, with the exception of the third entry, which is bounded by $c_1\delta$.

We want to write the intersection $V\cap C$ as a union of $N'$ connected analytic manifolds. It follows from Proposition \ref{prop} and the comments preceding it that this $N'$ will then be a bound for the $N$ in the statement. Because the projection from $\X_S$ to the maximal compact subgroup is an analytic isomorphism, it is enough to show that $V_S$ is the union of some number, $N''$, of connected analytic manifolds, and to give a bound on $N''$. (To get a bound for $N'$ we then have to multiply $N''$ by $2^g$.)

Since $V_S$ is a piecewise semi-pfaffian set, with the bounds on format mentioned above, it is a union of $c_1$ elementary semipfaffian sets, $\X_\nu$ say. These elementary semi-pfaffian sets have format entries bounded by $c_1$, again expect the $\beta$-entry which is bounded by $c_1\delta$. It is enough to show that each of these elementary semi-pfaffian sets can be written as a union of a controlled number of connected analytic manifolds. This we will do using the Gabrielov-Vorobjov stratification theorem (\cite[Theorem 2, page 82]{GVstrat}). We apply this theorem to each $\X_\nu$ in turn. For each $\nu$ the theorem gives a stratification of $\X_\nu$ into at most $c_2\delta^{c_3}$ smooth strata. Moreover, each of these strata has format bounded (entrywise) by $(c_6\delta^{c_7},c_5,c_5,c_5,c_6\delta^{c_7})$. (This means that the number of equations and the degrees of the polynomials in the pfaffian functions are bounded by $c_6\delta^{c_7}$, everything else by $c_5$.) To finish it is enough to bound the number of connected components of these strata. To do this, we add extra variables to remove the inequalities involved in the definitions of the strata. So, for each inequality of the form $g(x)>0$ we add variables $y,z$ and equations $g(x)-y^2=0, yz=1$. So we add $2\cdot c_5$ variables, and $2\cdot c_5$ equations, and write our stratum as the projection of a zero set of set of pfaffian functions in the higher-dimensional space given by the extra variables. By Khovanski's theorem (see \cite[Corollary 3.3]{GV}), this pfaffian set will have at most $c_8\delta^{c_9}$ connected components. Combining the estimates gives our bound:
\[
N\le c_1\cdot c_2 \cdot
c_8 \cdot\delta^{c_3+c_9}.
\]
%(I think $c_3$ and $c_9$ are probably something like $cg^{c'g}$ but I'm not sure.) \textit{Can you check that? We can add this in brackets but leave the computations out. }
\end{proof}
If the computation is carried out using the more explicit bounds given by Khovanskii, and by Gabrielov and Vorobjov, then it can be shown that both the constants in the exponent of $\delta$ have the form $(cg)^{c'g}$ for absolute effective constants $c$ and $c'$.  \\

\textbf{Funding}\\
 The main bulk of this work was done while the second author was at the University of Oxford as a fellow of the Swiss National Science Foundation. He is grateful for the hospitality of the institute and the generosity of the  Swiss National Science Foundation.  Both authors thank the Engineering and Physical Sciences Research Council for support under grant  EP/N007956/1.\\

\textbf{Acknowledgments}\\
The second author would like to thank his PhD advisor David Masser for his patience while answering the authors questions on maximal compact subgroups. He would also like to thank Daniel Bertrand for his helpful comments on the exact formulation of his transcendence result. Both authors thank Philipp Habegger for helpful discussions as well as Pietro Corvaja and Umberto Zannier for their interest in our work. Finally, we are most grateful to the referees for their many comments and suggestions, which have significantly improved the paper.

\end{document}